\documentclass[11pt]{amsart}

\usepackage{amssymb,amsmath,amsfonts}
\usepackage{bm}
\usepackage{enumerate}
\usepackage{array}

\allowdisplaybreaks

\newcommand{\RR}{\mathbb{R}}

\newcommand{\ZZ}{\mathbb{Z}}

\newtheorem*{thm*}{Theorem}
\newtheorem{thm}{Theorem}

\newtheorem{prop}[thm]{Proposition}

\def\N{{\mathbb N}}

\def\R{{\mathbb R}}
\def\RR{{\mathbb R}}
\def\TT{{\mathbb T}}

\def\Z{{\mathbb Z}}
\def\ZZ{{\mathbb Z}}

\def\scrG{{\mathcal G}}

\def\scrL{{\mathcal L}}

\def\sgn{\operatorname{sgn}}

\def\GamG{\Gamma\backslash G}

\def\SL{\operatorname{SL}}

\def\SO{\operatorname{SO}}

\numberwithin{equation}{section}

\begin{document}

\title[Higher dimensional gap theorems for the maximum metric]{Higher dimensional gap theorems\\ for the maximum metric}
\author{Alan Haynes, Juan J. Ramirez}

\thanks{AH: Research supported by NSF grant DMS 2001248.\\
\phantom{aa.}MSC 2020: 11J71, 37A45.}


\begin{abstract}
Recently, the first author together with Jens Marklof studied generalizations of the classical three distance theorem to higher dimensional toral rotations, giving upper bounds in all dimensions for the corresponding numbers of distances with respect to any flat Riemannian metric. In dimension two they proved a five distance theorem, which is best possible. In this paper we establish analogous bounds, in all dimensions, for the maximum metric. We also show that in dimensions two and three our bounds are best possible.
\end{abstract}

\maketitle


\section{Introduction}\label{sec.Intro}

Suppose that $d\in\N$, let $\scrL$ be a unimodular lattice in $\R^d$, and define $\TT^d=\R^d/\scrL$. For each $\vec\alpha\in\R^d$ and $N\in\N$ let $S_N=S_N(\vec\alpha,\scrL)$ denote the $d$-dimensional Kronecker sequence defined by
\begin{equation*}
	S_N=\{n\vec\alpha +\scrL : 1\le n\le N\}\subseteq\TT^d.
\end{equation*}
Given a metric $\mathrm{d}$ on $\R^d$ we define, for each $1\le n\le N$,
\begin{equation}\label{eqn.DeltaDef}
\delta^\mathrm{d}_{n,N}=\min\{\mathrm{d}(n\vec\alpha,m\vec\alpha+\vec\ell)>0:1\le m\le N,~\vec\ell\in\scrL\}.
\end{equation}
The quantity $\delta^\mathrm{d}_{n,N}$ is the smallest positive distance in $\R^d$ from $n\vec\alpha$ to an element of the set $S_N+\scrL$. As a natural generalization of the well known three distance theorem \cite{Sos1957,Sos1958,Sura1958,Swie1959}, we are interested in understanding, for each $\vec\alpha$ and $N$, the number
\begin{equation*}
g_N^\mathrm{d}=g_N^\mathrm{d}(\vec\alpha,\scrL)=|\{\delta^\mathrm{d}_{n,N}:1\le n\le N\}|
\end{equation*}
of distinct values taken by $\delta^\mathrm{d}_{n,N}$, for $1\le n\le N$. We will focus our discussion on two metrics: the Euclidean metric (for which we will write $\delta^\mathrm{d}_{n,N}=\delta_{n,N}$ and $g_N^\mathrm{d}=g_N$), and the maximum metric (for which we will write $\delta^\mathrm{d}_{n,N}=\delta^*_{n,N}$ and $g_N^\mathrm{d}=g_N^*$). To be clear, by the maximum metric on $\R^d$ we mean the metric defined by
\begin{equation*}
	\mathrm{d}(\vec x,\vec y)=\max_{1\le i\le d}|x_i-y_i|.
\end{equation*}

For the case of the Euclidean metric it is known that, for any $\scrL, \vec\alpha,$ and $N$,
\begin{equation*}
	g_N(\vec\alpha,\scrL) \le 
	\begin{cases}
		3&\text{if}~d=1,\\
		5&\text{if}~d=2,\\ 
		\sigma_d+1 &\text{if}~d\ge 3  ,
	\end{cases}
\end{equation*}
where $\sigma_d$ is the maximum number of non-overlapping spheres of radius one in $\R^d$ which can be arranged so that they all touch the unit sphere in exactly one point ($\sigma_d$ is also known as the \textit{kissing number} for $\R^d$). The bound for $d=1$, in which case the Euclidean and maximum metrics coincide, is a slightly modified version of the three distance theorem (the classical three distance theorem considers the number of `one-sided' gaps, which can in general be greater than $g_N$). The bounds for $d\ge 2$ were recently established in \cite{HaynMark2020b}. For $d=1$ and $2$ there are examples of $\scrL, \vec\alpha,$ and $N$ for which the upper bounds above are actually obtained (see the introduction of \cite{HaynMark2020b}), therefore those bounds are best possible. For $d\ge 3$ the upper bounds above are 13, 25, 46, 79, 135, 241, 365, 555, etc. These are probably far from best possible, but improving them substantially may require new ideas.

In this paper, motivated both by historical precedent and by questions which were asked of us after the publication of \cite{HaynMark2020b}, we will show how the machinery from that paper can be used to easily bound the corresponding quantity $g_N^*$ for the maximum metric, and even to obtain the best possible bounds in dimensions $d=2$ and $3$. To our knowledge, the only known result about this problem is due to Chevallier \cite[Corollaire 1.2]{Chev1996}, who showed that $g_N^*\le 5$ when $d=2$ and $\scrL=\Z^2$. Chevallier also gave an example in this case (see remark at end of \cite[Section 1]{Chev1996}) for which $g_N^*=4$. We will prove the following theorem.
\begin{thm}\label{thm.GapsBd1}
For any $d,\scrL, \vec\alpha,$ and $N$, we have that
\begin{equation*}
g_N^*(\vec\alpha,\scrL)\le 2^d+1.
\end{equation*}
Furthermore, when $d=2$ or $3$ this bound is, in general, best possible.
\end{thm}
To prove Theorem \ref{thm.GapsBd1} we will first realize the quantity $g_N^*$ as the value of a function $\scrG$ defined on the space $\SL (d+1,\Z)\backslash\SL (d+1,\R)$ of unimodular lattices in $\R^{d+1}$. This part of the proof, carried out in Section \ref{sec.Latt}, is exactly analogous to the development in \cite{HaynMark2020} and \cite{HaynMark2020b}, which in turn is an extension of ideas originally presented by Marklof and Str\"{o}mbergsson in \cite{MarkStro2017}. In Section \ref{sec.Proofs} we will use a simple geometric argument to bound $\scrG(M)$, when $M$ is an arbitrary unimodular lattice in $\R^d$, and for $d=2$ and $3$ we will give examples of $\scrL, \vec\alpha,$ and $N$ for which our upper bounds are attained. Such examples, especially when $d=3$, appear to be quite difficult to find.

Finally we remark that for $d=2$ the conclusions of Theorem \ref{thm.GapsBd1} also hold for the Manhattan metric (i.e. the $\ell^1$ metric on $\R^d$). To see this, observe that the unit ball for this metric is a rotated and homothetically scaled copy of the unit ball for the maximum metric. It follows that, if $d=2$ and if $\mathrm{d}$ is the Manhattan metric on $\R^2$, then there is a matrix $R\in\SO (2,\R)$ (rotation by $\pi/4$) with the property that, for every $\scrL, \vec\alpha,$ and $N$,
\begin{equation*}
		g_N^{\mathrm{d}}(\vec\alpha,\scrL)=g_N^*\left(R\vec\alpha,R\scrL\right).
\end{equation*}
Therefore $g_N^{\mathrm{d}}\le 5$ for this metric also, and this bound is best possible. \vspace*{10bp}

\noindent Acknowledgments: We would like to thank Jens Marklof and Nicolas Chevallier for bringing this problem to our attention, and for helpful comments. We also thank the referee for their feedback, and for carefully reading our paper. This project is part of the second author's undergraduate senior research project at the University of Houston.

\section{Lattice formulation of the problem}\label{sec.Latt}

As mentioned above, the observations in this section are very similar to those in \cite[Section 2]{HaynMark2020b}. Therefore we will omit some of the details, which are explained in full in that paper.

Let $|\cdot|_\infty$ denote the maximum norm on $\R^d$. By a linear change of variables in the definition \eqref{eqn.DeltaDef}, we have that
\begin{equation*}
\begin{split}
\delta_{n,N}^* = \min\{| k\vec\alpha + \vec\ell |_\infty>0 : -n< k< N_+-n,~ \vec\ell\in\scrL \} ,
\end{split}
\end{equation*}
where $N_+:=N+\tfrac12$.
Choose $M_0\in\SL(d,\RR)$ so that $\scrL=\ZZ^d M_0$, and let
\begin{equation*}
A_N(\vec\alpha)=A_N(\vec\alpha,\scrL)=\begin{pmatrix} 1 & 0 \\ 0 & M_0  \end{pmatrix} \begin{pmatrix} 1 & \vec{\alpha} \\ 0 & \bm{1}_d  \end{pmatrix} \begin{pmatrix} N^{-1} & 0 \\ 0 & N^{1/d}\bm{1}_d\end{pmatrix}.
\end{equation*}
Then, for all $1\leq n\leq N$, we have that
\begin{multline*}
\delta_{n,N}^* = N_+^{-1/d} \min\bigg\{| \vec v |_\infty>0 :  (u,\vec v)\in\ZZ^{d+1} A_{N_+}(\vec\alpha),~ -\frac{n}{N_+}< u < 1-\frac{n}{N_+} \bigg\} .
\end{multline*}
Now write $G=\SL(d+1,\RR)$ and $\Gamma=\SL(d+1,\ZZ)$ and, for $M\in G$ and $t\in (0,1)$, define
\begin{equation*}
F(M,t) = \min\big\{| \vec v |_\infty>0 :  (u,\vec v)\in\ZZ^{d+1} M, ~-t< u < 1-t \big\} .
\end{equation*}
It follows from the proof of \cite[Proposition 1]{HaynMark2020b} that $F$ is well-defined as a function from $\GamG\times (0,1)$ to $\RR_{>0}$. It is also clear that
\begin{equation*}
\delta_{n,N}^*= N_+^{-1/d} F\left(A_{N_+}(\vec\alpha),\frac{n}{N_+} \right) .
\end{equation*}
Given $M\in G$, a bounded region of $\R^{d+1}$ can contain only finitely many points of the lattice $\Z^{d+1}M$. This implies, after a short argument, that the function $F(M,t)$ can only take finitely many values as $t$ varies over $(0,1)$. We denote this finite number by
\begin{equation*}
	\scrG(M)=|\{ F(M,t) \mid 0<t<1\}|,
\end{equation*}
and for $N\in\N$ we also write
\begin{equation*}
\scrG_{N}(M)=|\{ F(M,\tfrac{n}{N_+}) \mid 1\leq n \leq N\}|.
\end{equation*}
It follows from the definitions above that
\begin{equation}\label{eqn.g_NBnd}
g_N^* = \scrG_{N}(A_{N_+}(\vec\alpha)) \leq \scrG(A_{N_+}(\vec\alpha)).
\end{equation}
This is the key connection which we will use in our proof of Theorem \ref{thm.GapsBd1}.

\section{Proof of Theorem \ref{thm.GapsBd1}}\label{sec.Proofs}
To prove the first part of Theorem \ref{thm.GapsBd1}, in light of \eqref{eqn.g_NBnd} it is sufficient to show that, for any $M\in\GamG,$
\begin{equation}\label{eqn.G(M)Bd}
\scrG(M)\le 2^d+1.
\end{equation}
Suppose that $M\in\GamG$ and choose vectors $(u_1,\vec{v}_1),\ldots,(u_K,\vec{v}_K)\in\Z^{d+1}M$, with $K=\scrG(M)$, so that the following conditions hold:\vspace*{3bp}
\begin{itemize}
	\item $0<|\vec{v}_1|_\infty<|\vec{v}_2|_\infty<\cdots <|\vec{v}_K|_\infty$.\vspace*{3bp}
	\item For each $t\in (0,1)$, there exists an $1\le i\le K$ such that $|\vec{v}_i|_\infty=F(M,t)$.\vspace*{3bp}
	\item For each $1\le i\le K$, there exists a $t\in(0,1)$ such that $-t<u_i<1-t$ and $|\vec v_i|_\infty=F(M,t)$.\vspace*{3bp}
\end{itemize}
Note that each $u_i$ lies in the interval $(-1,1)$. We make the following basic observation.
\begin{prop}\label{prop.SmallT}
	If $(u,\vec{v})\in\Z^{d+1}M$ satisfies $|u|<1/2$, then
	\begin{equation*}
		F(M,t)\le |\vec v|_\infty,
	\end{equation*}
for all $0<t<1$.	
\end{prop}
\begin{proof}
If $u\in[0,1/2)$ then for any $0<t<1-u$, we have that $F(M,t)\le|\vec{v}|_\infty$. Noting that $(-u,-\vec v)\in\Z^{d+1}M$, we see that this inequality also holds for any $t$ satisfying $u<t<1$. Since $u<1/2$, we conclude that $F(M,t)\le |\vec{v}|_\infty$ for all $0<t<1$. The case when $u\in(-1/2,0]$ follows from the same argument.
\end{proof}
Next we use geometric information to place restrictions on the vectors $(u_i,\vec v_i)$. This is where we will use the fact that we are working with the maximum norm.
\begin{prop}\label{prop.NoSameOrth}
For $1\le i<j\le K$, if $\sgn(u_i)\vec v_i$ and $\sgn(u_j)\vec v_j$ lie in the same orthant of $\R^d$, then $j=K$.
\end{prop}
\begin{proof}
If $|u_j|<1/2$ then by the previous proposition we have that $F(M,t)\le|\vec v_j|_\infty$ for all $0<t<1$, which implies that $j=K$. Therefore suppose that $|u_j|\ge 1/2$. Then, by Proposition \ref{prop.SmallT} again, this forces $|u_i|\ge 1/2$.

If $\sgn(u_i)=\sgn(u_j)$ and if $\vec v_i$ and $\vec v_j$ lie in the same orthant, then $|u_i-u_j|<1/2$, and
\begin{equation*}
0<|\vec v_i-\vec v_j|_\infty\le \max\left\{|\vec v_i|_\infty,|\vec v_j|_\infty\right\}=|\vec v_j|_\infty.
\end{equation*}
Since $(u_i-u_j,\vec v_i-\vec v_j)\in\Z^{d+1}M$, it follows from Proposition \ref{prop.SmallT} that
\begin{equation}\label{eqn.SameOrthBd}
F(M,t)\le |\vec v_i-\vec v_j|_\infty\le |\vec v_j|_\infty
\end{equation}
for all $0<t<1$. Therefore we conclude that $j=K.$

Similarly, if $\sgn(u_i)=-\sgn(u_j)$ and if $\vec v_i$ and $-\vec v_j$ lie in the same orthant, then
$|u_i+u_j|<1/2$, and
\begin{equation*}
	0<|\vec v_i+\vec v_j|_\infty\le \max\left\{|\vec v_i|_\infty,|\vec v_j|_\infty\right\}=|\vec v_j|_\infty.
\end{equation*}
Since $(u_i+u_j,\vec v_i+\vec v_j)\in\Z^{d+1}M$, this again implies that \eqref{eqn.SameOrthBd} holds, and we conclude that $j=K$.
\end{proof}
By Proposition \ref{prop.NoSameOrth}, each of the vectors $\sgn(u_i)\vec v_i$, for $1\le i\le K-1$, must lie in a different orthant of $\R^d$. This immediately gives the bound in \eqref{eqn.G(M)Bd}, and therefore completes the proof of the first part of Theorem \ref{thm.GapsBd1}.

Finally, we give examples with $d=2$ and 3 for which the bound in Theorem \ref{thm.GapsBd1} is attained. In what follows, for $\vec x\in\R^d$ we write
\begin{equation*}
	\|\vec x\|=\min\{|\vec x-\vec \ell|_\infty : \ell\in\scrL\}.
\end{equation*}

For $d=2$ take $\scrL=\Z^2,~\vec\alpha=(157/500,-23/200),$ and $N=11$. Then we have that
\begin{align*}
\delta^*_{1,N}=\|10\vec\alpha\|=\frac{3}{20},\quad\delta^*_{2,N}=\|9\vec\alpha\|=\frac{87}{500},\quad\delta^*_{4,N}=\|7\vec\alpha\|=\frac{99}{500},
\end{align*}
\begin{align*}
	\delta^*_{5,N}=\|6\vec\alpha\|=\frac{31}{100},\quad\text{and}\quad\delta^*_{6,N}=\|\vec\alpha\|=\frac{157}{500},
\end{align*}
therefore $g^*_N(\alpha,\scrL)=5$.

For $d=3$ take $\scrL=\Z^3,~\vec\alpha=(-157/10000, -742/3125, -23/400),$ and $N=73$. Then we have that
\begin{align*}
\delta^*_{1,N}=\|72\vec\alpha\|=\frac{7}{50},\quad\delta^*_{2,N}=\|71\vec\alpha\|=\frac{443}{3125},\quad\delta^*_{5,N}=\|68\vec\alpha\|=\frac{456}{3125},
\end{align*}
\begin{align*}
\delta^*_{6,N}=\|67\vec\alpha\|=\frac{59}{400},\quad\delta^*_{18,N}=\|55\vec\alpha\|=\frac{13}{80},\quad \delta^*_{19,N}=\|54\vec\alpha\|=\frac{557}{3125},
\end{align*}
\begin{align*}
\delta^*_{22,N}=\|51\vec\alpha\|=\frac{1993}{10000},\quad\delta^*_{23,N}=\|50\vec\alpha\|=\frac{43}{200},\quad\text{and}\quad\delta^*_{24,N}=\|4\vec\alpha\|=\frac{23}{100},
\end{align*}
therefore $g^*_N(\alpha,\scrL)=9$.

This completes the proof of our main result. We conclude with a couple of remarks. First of all, we note that the proof that we have given here (in particular, establishing the optimal bounds in dimensions 2 and 3) is simpler than the proof of the corresponding result for the Euclidean metric given in \cite{HaynMark2020b}. For the proof of the optimal bounds in dimensions 2 and 3 here, we only needed Proposition \ref{prop.NoSameOrth}, which plays a similar role for the maximum metric as \cite[Proposition 2]{HaynMark2020b} does for the Euclidean metric. However, for the Euclidean metric, \cite[Proposition 2]{HaynMark2020b} is not enough to establish the optimal bound in dimension 2, which is the reason for the additional geometric arguments given in \cite[Section 5]{HaynMark2020b}. Of course, a less precise explanation for this is that cube packing is easier and more efficient than sphere packing.

Finally, for $d\ge 4,$ it seems likely that the upper bound of Theorem \ref{thm.GapsBd1} is too large. In fact, for $d=4$, we have not found any examples so far with $g_N^*>9$. Establishing optimal upper bounds in these cases is an interesting open problem.

\vspace{.15in}

{\footnotesize
\noindent
AH, JR: Department of Mathematics, University of Houston,\\
Houston, TX, United States.\\
haynes@math.uh.edu, juan.ramirez789456@gmail.com

}

\end{document}